\documentclass[]{amsart}

\usepackage{amssymb, amsmath, amsthm, amsfonts, tikz-cd,stmaryrd,mathtools}
\usepackage[all]{xy}
\usepackage{tikz}
\usetikzlibrary{positioning} %Needed for the basic examples
\usetikzlibrary{arrows,decorations.pathmorphing,shapes} %Needed for the crazy example
\title{Toward partial Verma functors of $\mathcal{U}^{[r]}(\mathfrak{g})$ and related results}
\author{Pablo Boixeda Alvarez}

\newtheorem{thm}{Theorem}
\newtheorem{conj}{Conjecture}
\newtheorem{prop}{Proposition}
\newtheorem{lem}{Lemma}

\newtheorem{rmk}{Remark}\theoremstyle{remark}
\begin{document}

\begin{abstract}
This note extends the Steinberg Tensor product theorem from the Frobenius kernel $G_{(r)}$ to the deformation $\mathcal{U}^{[r]}(\mathfrak{g})$ of its distribution algebra.
As a Corollary we proof some conjectures from \cite{Wes}. Further it describes the graded representation theory of $\mathcal{U}^{[r]}(\mathfrak{g})$ at a generic semisimple p-central character as a twist of the category of graded $G_{(r-1)}$-modules. We conclude by explaining this relation for $SL_2$ explicitely as well as computing the center in this case.
\end{abstract}
\maketitle
\section{Introduction}
Let $\mathfrak{g}$ be a semisimple Lie algebra and consider the algebra $\mathcal{U}^{[r]}(\mathfrak{g})$ of divided powers up to degree $p^{r-1}$. Let $\mathfrak{b}$ be a Borel subalgebra.\\
$\mathcal{U}^{[r]}(\mathfrak{g})$ has a central subalgebra isomorphic to $k[\mathfrak{g}^{(r)}]$, which can be thought of as given by the $p^{r}$ divided powers multiplied by p, or in the case of nilpotent elements just the pth power of the $p^{r-1}$ divided power. For $\chi\in\mathfrak{g}^{(r)}$ denote by $\mathcal{U}_\chi^{[r]}(\mathfrak{g})$ the corresponding central reduction.\\
 There are maps $\mathcal{U}^{[r]}(\mathfrak{g})\rightarrow\mathcal{U}^{[r+1]}(\mathfrak{g})$ which factors through an inclusion $\mathcal{U}_0^{[r]}(\mathfrak{g})\hookrightarrow\mathcal{U}^{[r+1]}(\mathfrak{g})$
We can thus consider modules over $\mathcal{U}^{[R]}(\mathfrak{g})$ as modules over $\mathcal{U}^{[r]}(\mathfrak{g})$ if $R>r$.\\
Further we have the Frobenius map $F:\mathcal{U}^{[r]}(\mathfrak{g})\rightarrow \mathcal{U}^{[r-1]}(\mathfrak{g}^{(1)})$. If $V$ is a $\mathcal{U}^{[r-1]}(\mathfrak{g}^{(1)})$ module we denote by $V^{(1)}$ the same module considered as an $\mathcal{U}^{[r]}(\mathfrak{g})$ module. Similalry denote by $V^{(i)}$ the modules defined through pullback along $F^i$.\\
The goal of this small note is to classify the simple representations of $\mathcal{U}^{[r]}(\mathfrak{g})$ in terms of simple $\mathcal{U}(\mathfrak{g})$ representations and explain a relation between representations of $\mathcal{U}_\chi^{[r+1]}(\mathfrak{g})$ and $\mathcal{U}_0^{[r]}(\mathfrak{g})$, for a generic $\chi$. We conclude the paper with some explicit computations in $SL_2$ including a computation of the center.\\
\subsection{Acknowledgments} I want to thank Roman Bezrukavnikov for suggesting this project, for a lot of conversations about the topic and for giving me great suggestions to improve all my work.
\section{Steinberg Tensor Product for $\mathcal{U}^{[r]}(\mathfrak{g})$}
We define the baby Vermas. Let $\chi$ and $\mathfrak{b}$ be such that  $\chi|_\mathfrak{n}=0$, where $\mathfrak{n}$ denotes the nilpotent radical of $\mathfrak{b}$. Similarly choose $\mathfrak{h}$ a Cartan subalgebra of $\mathfrak{b}$. Given a $\mathcal{U}_\chi(\mathfrak{h})$ representation $M$ we can define a $\mathcal{U}_\chi(\mathfrak{b})$ representation by letting $\mathfrak{n}$ act by 0, which is well defined since $\chi(\mathfrak{n})=0$. Denote the set of $\lambda:\mathfrak{h}\rightarrow k$ a linear map, such that $\lambda(h)^p-\lambda(h)=\chi(h)^p$ by $\Lambda_\chi^{1}$ (here we use the identification $\mathfrak{g}=\mathfrak{g}^*$ using the Killing form, to evaluate $\chi(h)$). Similarly define $\Lambda_\chi^r$ as the set of characters of $\mathcal{U}^{[r]}_\chi(\mathfrak{h})$. We can define a representation $k_\lambda^\chi$ of $\mathcal{U}^{[r]}_\chi(\mathfrak{h})$, by acting through $\lambda\in\Lambda_\chi^r$. Denote by $Z^{r,\chi}_\lambda:=\mathcal{U}^{[r]}_\chi(\mathfrak{g})\otimes_{\mathcal{U}^{[r]}_\chi(\mathfrak{b})}k_ \lambda^\chi$. By well known results there is a unique simple quotient of this module, which we denote by $L^\chi_\lambda$. Denote by $\Lambda^r:=\Lambda^r_0$ and $L_\lambda:=L^0_\lambda$. Further note that there is a map $\Lambda^R_\chi\rightarrow\Lambda^r$ given by restriction and further this map splits cannonically $\Lambda^r\subset\Lambda^R$ for $\chi=0$, given by the characters that act by 0 on higher divided powers.

The following result follows immediately from \cite{Jan}, II 3.15, as the representations of $\mathcal{U}_0^{[r]}(\mathfrak{g})$ are the same as $G_{(r)}$ representation
\begin{prop}
	A simple module $L_\lambda$ of $\mathcal{U}_0^{[R]}(\mathfrak{g})$ for $\lambda\in \Lambda^r$ remains simple when restricted to $\mathcal{U}_0^{[r]}(\mathfrak{g})$ 
\end{prop}

\begin{thm}
	The simple representations of $\mathcal{U}^{[r]}(\mathfrak{g})$ are 
	$$L_{\lambda_1}\otimes L_{\lambda_2}^{(1)}...\otimes (L^{\chi}_{\lambda_r})^{(r-1)}$$
	for $\lambda_i\in \Lambda^1$ $i\neq r$ and $\lambda_r\in\Lambda^1_\chi$.
\end{thm}
\begin{proof}
	Let $V$ be a simple representation of $\mathcal{U}^{[r]}(\mathfrak{g})$. Let $L$ be a simple summand of the socle of $V$ as a $\mathcal{U}_0^{[r-1]}(\mathfrak{g})$ representation. By the above proposition $L$ extends to a simple $\mathcal{U}^{[r]}(\mathfrak{g})$ representation and we can consider the representation $W=Hom_{\mathcal{U}_0^{[r-1]}(\mathfrak{g})}(L,V)$, which factors as a $\mathcal{U}(\mathfrak{g}^{(r-1)})$ representation. By the choice of $L$ this is a non-zero representation. Further we can define a map $W\otimes L\rightarrow V$, which is injective, because $L$ is simple, and as $W$ is non-zero and $V$ is simple, we have this map is an isomorphism. Further as $V$ is simple it follows from this isomorphism that $W$ has to be simple. 
	Conversely assume $L$ is a simple as in the Proposition and let $W$ be a simple representation of $\mathcal{U}(\mathfrak{g}^{(r-1)})$, we will show that $W\otimes L$ is simple. Assume not and let $V\hookrightarrow W\otimes L$ be a simple subrepresentation. We apply the above construction to $V$. Restricted to $\mathcal{U}_o^{[r-1]}(\mathfrak{g})$, $W\otimes L$ is a direct sum of copies of $L$, so we have $V\cong W'\otimes L$. Then we have
	\begin{align*}
		Hom_{\mathcal{U}^{[r]}(\mathfrak{g})}(V,W\otimes L)&=Hom_{\mathcal{U}(\mathfrak{g}^{(r-1)})}(k,Hom_{\mathcal{U}_0^{[r-1]}(\mathfrak{g})}(W'\otimes L,W\otimes L))\\&=Hom_{\mathcal{U}(\mathfrak{g}^{(r-1)})}(k,Hom(W',W))\\&=Hom_{\mathcal{U}(\mathfrak{g}^{(r-1)})}(W',W)
	\end{align*}
	And as this is a non-zero vector space and $W'$ and $W$ are simples, we get $W'$ is isomorphic to $W$, hence we must have $V=W$ and hence $W$ is already simple. Now the results follows by induction.
\end{proof}

Next we prove the Conjectures stated in \cite{Wes}. 
%First we introduce the notation used there. Denote by $\widehat{\mathcal{U}^{[r]}_\chi (\mathfrak{b})}$, the subalgebra of $\mathcal{U}_\chi^{[r]}(\mathfrak{g})$ generated by $\mathcal{U}_\chi^{[r-1]}(\mathfrak{g})$ and $\mathcal{U}_\chi^{[r]}(\mathfrak{b})$, for $\mathfrak{b}$ such that $\chi|_{\mathfrak{n}}=0$.

\begin{thm}
	Let $N$ be an irreducible $Dist(G_{(r-1)})$-module with corresponding weight $\lambda_{r-1}\in\Lambda^{r-1}_0$. Let $\widehat{\mathcal{U}^{[r]}_\chi (\mathfrak{b})}$
	be the subalgebra of $\mathcal{U}_\chi^{[r]}(\mathfrak{g})$ generated by $Dist(G_{(r-1)})$ and $\mathcal{U}^{[r]}_\chi (\mathfrak{b})$. Then each extension of $\lambda_{r-1}$ to $\lambda_r\in\Lambda^r_\chi$ determines an irreducible $\widehat{\mathcal{U}^{[r]}_\chi (\mathfrak{b})}$-module structure on the $Dist(G_{(r-1)})$-module N, and every irreducible $\widehat{\mathcal{U}^{[r]}_\chi (\mathfrak{b})}$-module restricts to an irreducible $Dist(G_{(r-1)})$-module.
\end{thm}
\begin{proof}
	By the above any irreducible of $Dist(G_{(r-1)})$ extends to $L_{\lambda_{r-1}}$ a $\mathcal{U}_0^{[r]}(G)$-module and so restrict to a $\widehat{\mathcal{U}^{[r]}_0 (\mathfrak{b})}$-module. Now we have a map $$\widehat{\mathcal{U}^{[r]}_\chi (\mathfrak{b})}\rightarrow \mathcal{U}_\chi(\mathfrak{b})$$
	so as $\chi|_\mathfrak{n}=0$ we get a module $k_{\lambda_r}^\chi$ of $\widehat{\mathcal{U}^{[r]}_\chi (\mathfrak{b})}$, by pullback along the above map. Now consider
	$$N=L_{\lambda_r}\otimes k_{\lambda_r}^\chi$$
	will satisfy the necessary properties.\\
	To prove that all irreducibles are of the above form, we note that the above proof of Steinberg's tensor product theorem, extends to $\widehat{\mathcal{U}^{[r]} (\mathfrak{b})}$ and hence we just need to classify the irreducible modules of $\mathcal{U}_\chi (\mathfrak{b}^{(r-1)})$ and as $\mathfrak{n}$ acts nilpotently the irreducible representations are given by $k_\lambda^\chi$ for $\lambda\in\Lambda^1_\chi$
\end{proof}

Note now that $Z^{[r]}_\chi(N,\lambda_r):=\mathcal{U}_\chi^{[r]}(\mathfrak{g})\otimes_{\widehat{\mathcal{U}^{[r]}_\chi (\mathfrak{b})}}N=(Z^\chi_{\lambda_r})^{(r)}\otimes L_{\lambda_{r-1}}$. Now the fact that all irreducibles are the image of these modules follows from the Steinberg tensor product theorem and the same result for the enveloping algebra.
\\
Further assume that $\chi$ is semisimple and assume $\chi\in\mathfrak{h}$ for some Cartan subalgebra $\mathfrak{h}$. Let $T$ be the torus corresponding to that Cartan subalgebra. Note that we have a $T$ action on this algebra as $\chi$ is fixed by the action of $T$ and so we can defined a graded category $\mathcal{U}^{[r]}_0(\mathfrak{g})-mod^{gr}$ of modules of $\mathcal{U}^{[r]}_0(\mathfrak{g})$ with a grading given by $T$ compatible with the $T$ action on $\mathcal{U}^{[r]}_0(\mathfrak{g})$. We extend the above proven conjecture to the following.
\begin{conj}
	The above assignment on simples and a weight $\lambda$ given by $(L,\lambda)\mapsto Z^{(r)}_\chi(\lambda)\otimes L$ (or further similarly defined on modules that come under restriction of $\mathcal{U}^{[r]}(\mathfrak{g})$) extends to a functor
	$$F:\mathcal{U}^{[r]}_0(\mathfrak{g})-mod^{gr}\rightarrow\mathcal{U}^{[r+1]}_\chi(\mathfrak{g})-mod^{gr}$$
\end{conj}

To justify this conjecture we will see that in the case for $\chi$ generic, there is a clear relation.
\section{Towards partial Verma functors}
%, which in the case of $Sl_2$ is an equivalence.\\

%\begin{thm}
%	For $\chi$ regular semisimple, we have an equivalence of categories
%	$$F:\mathcal{U}^{[r]}_0(\mathfrak{g})-mod^{gr}\rightarrow\mathcal{U}^{[r+1]}_\chi(\mathfrak{g})-mod^{gr}$$
%\end{thm}
%\begin{proof}
To understand this relation we need to understand the algebras governing these two categories, $\mathcal{U}^{[r]}_0(\mathfrak{g})-mod^{gr}$ and $\mathcal{U}^{[r+1]}_\chi(\mathfrak{g})-mod^{gr}$ for a regular semisimple $\chi$.That is we will consider projective generators $P$ and $P'$ for each category respectively, with one indecomposable projective summand for each simple and we will prove a relation between $End(P)$ and $End(P')$. Namely we prove the following result:
\begin{thm}
	The algebras $End(P)$ has the structure of a $G$ module.There is an isomorphism of vector spaces $End(P)\cong End(P')$. There are elements $c^\chi_\lambda\in\mathcal{U}(\mathfrak{g})\otimes \mathcal{U}(\mathfrak{g}) $, such that the product of $End(P)$ and $End(P')$ for maps between indecomposable projective are related under the above isomorphism by multiplying $End(P)\otimes End(P)$ by $c^\chi_\lambda$ using the structure of $G$-modules and then using the product of $End(P)$. Here the choice of $c^\chi_\lambda$ depends on the grading of the indecomposable projectives used to described the above maps.
\end{thm}
The rest of this section will be a proof of the above theorem, as well as explaining precisely what the $c_\lambda$ is.\\
To do this we note that the construction in the previous section works for any module, that is given by restriction from a $\mathcal{U}^{[r+1]}(\mathfrak{g})$ module. By \cite{Jan1} for p large enough the projective cover $_rP_\lambda$ of $L_\lambda$ as $\mathcal{U}^{[r]}_0(\mathfrak{g})$ modules for $\lambda\in\Lambda^r$ extends to a module of $\mathcal{U}^{[r+1]}_0(\mathfrak{g})$ in a canonical way. Further denote by $_rP^\chi_\lambda$ the projective cover of $L^\chi_\lambda$ as $\mathcal{U}^{[r]}_\chi(\mathfrak{g})$ modules.
	\begin{lem}
		For $\lambda_0\in\Lambda^r$ and $\lambda_1\in\Lambda$ and $\chi$ a semisimple character we have an isomorphism of graded $\mathcal{U}^{[r+1]}_\chi(\mathfrak{g})$ modules
		$$ _{r+1}P^\chi_{\lambda_0+p^r\lambda_1}=\textbf{} _rP_{\lambda_0}\otimes (\textbf{}_1P^\chi_{\lambda_1})^{(r)}$$
	\end{lem}
	\begin{proof}
		This follows by generalizing an argument in \cite{Jan}, just as the Steinberg tensor product theorem generalizes to general p-character.
	\end{proof} 
	From this we get that in the case of $\chi$ generic the indecomposable projectives of $\mathcal{U}^{[r+1]}_\chi(\mathfrak{g})$ are given by $(Z^\chi_{\lambda_1})^{(r)}\otimes\textbf{} _rP_{\lambda_0}$\\
	As $_rP_{\lambda}$ extends to a $\mathcal{U}^{[r+1]}_0(\mathfrak{g})$ module, we have that $Hom_{\mathcal{U}^{[r]}_0(\mathfrak{g})}(_rP_{\lambda},{}_rP_{\lambda'})$ can be given the structure of a $\mathcal{U}(\mathfrak{g})$ module. Lets call this module $V^{(\lambda,\lambda')}$ then composition is given by a $\mathcal{U}(\mathfrak{g})$ module map $V^{(\lambda,\lambda')}\otimes V^{(\lambda',\lambda'')}\rightarrow V^{(\lambda,\lambda'')}$. We have then the following isomorphisms:
	$$Hom_{\mathcal{U}^{[r+1]}_\chi(\mathfrak{g})}((Z^\chi_\mu)^{(r)}\otimes{}_rP_\lambda,(Z^\chi_{\mu'})^{(r)}\otimes{}_rP_{\lambda'})=Hom_{\mathcal{U}_\chi(\mathfrak{g^{(r-1)}})}(Z^\chi_\mu,Z^\chi_{\mu'}\otimes V^{(\lambda,\lambda')})$$
	Now under the above isomorphism, we get the composition is given as follows:
	\begin{align*}
		Hom_{\mathcal{U}_\chi(\mathfrak{g^{(r-1)}})}(Z^\chi_\mu,Z^\chi_{\mu'}\otimes V^{(\lambda,\lambda')})\otimes Hom_{\mathcal{U}_\chi(\mathfrak{g^{(r-1)}})}(Z^\chi_{\mu'},Z^\chi_{\mu''}\otimes V^{(\lambda',\lambda'')})\\\rightarrow Hom_{\mathcal{U}_\chi(\mathfrak{g^{(r-1)}})}(Z^\chi_\mu,Z^\chi_{\mu''}\otimes V^{(\lambda,\lambda')}\otimes V^{(\lambda',\lambda'')})\\\rightarrow Hom_{\mathcal{U}_\chi(\mathfrak{g^{(r-1)}})}(Z^\chi_{\mu},Z^\chi_{\mu''}\otimes V^{(\lambda',\lambda'')})\\
	\end{align*}
	Where the first map above is given by $f\otimes g\mapsto (g\otimes id)\circ f$ and the second map comes from the above given map $V\otimes W\rightarrow U$. Now the indecomposable projectives in $\mathcal{U}^{[r]}_0(\mathfrak{g})-mod^{gr}$ are given by $k_{p^r\mu}^\chi\otimes {}_rP_\lambda$ for $\lambda\in\Lambda^r$ and $\mu\in\Lambda$ and the indecomposable projectives of $\mathcal{U}^{[r+1]}_\chi(\mathfrak{g})-mod^{gr}$ are $(Z^\chi_\mu)^{(r)}\otimes{}_rP_\lambda$ for $\mu$ and $\lambda$ as above. Then the homomorphism in the graded categories are given by 
	$$Hom_{\mathcal{U}^{[r]}_0(\mathfrak{g})-mod^{gr}}(k_{p^r\mu}^\chi\otimes {}_rP_\lambda,k_{p^r\mu'}^\chi\otimes {}_rP_{\lambda'})\cong V^{(\lambda,\lambda')}_{\mu-\mu'}$$
	and 
	\begin{align*}
		Hom_{\mathcal{U}^{[r+1]}_\chi(\mathfrak{g})-mod^{gr}}((Z^\chi_\mu)^{(r)}\otimes{}_rP_\lambda,(Z^\chi_{\mu'})^{(r)}\otimes{}_rP_{\lambda'})\\\cong Hom_{\mathcal{U}_\chi(\mathfrak{g^{(r-1)}})-mod^{gr}}(Z^\chi_\mu,Z^\chi_{\mu'}\otimes V^{(\lambda,\lambda')})
	\end{align*}
	Here for a $\mathcal{U}_\chi(\mathfrak{g})$ module $V$, $V_\lambda$ is the $\lambda$ weight space of $V$. 
	Here the composition in the second category is given as explained above. And the first composition is given by $V^{(\lambda,\lambda')}_{\mu-\mu'}\otimes V^{(\lambda',\lambda'')}_{\mu'-\mu''}\rightarrow (V^{(\lambda,\lambda')}\otimes V^{(\lambda',\lambda'')})_{\mu-\mu''}\rightarrow V^{(\lambda,\lambda'')}_{\mu-\mu''}$. Here the first map is just that obvious inclusion and the second is given by the above composition product of the $\mathcal{U}(\mathfrak{g})$ modules $V^{(\lambda',\lambda'')}$\\
	We will now give an isomorphism $$Hom_{\mathcal{U}_\chi(\mathfrak{g^{(r-1)}})-mod^{gr}}(Z^\chi_\mu,Z^\chi_{\mu'}\otimes V^{(\lambda,\lambda')})\cong V^{(\lambda,\lambda')}_{\mu-\mu'}$$
	Such that the two products are related by a certain twist by an action of an element in  $\mathcal{U}(\mathfrak{g})\otimes \mathcal{U}(\mathfrak{g})$.
	Now we will try to find an isomorphism between these two vector spaces respecting composition. To do this we will use $Z^\chi_\mu\cong Ind_\mathfrak{b}^\mathfrak{g}(k^\chi_\mu)$ and so $Z^\chi_\mu\otimes V\cong Ind_\mathfrak{b}^\mathfrak{g}(k^\chi_\mu\otimes Res_\mathfrak{b}^\mathfrak{g}(V))$. Consider $(Z^\chi_\mu)^*\otimes Z^\chi_\mu\cong \mathcal{U}_0(\mathfrak{n})\otimes\mathcal{U}_0(\mathfrak{n}_-)$ this is given as baby Verma for generic characters are both isomorphic to $\mathcal{U}_0(\mathfrak{n})$ and $\mathcal{U}_0(\mathfrak{n}_-)$ given by choosing a lowest weight or highest weight vector. These can be chosen compatibly so that the pairing gives 1. Consider the element $x\in (Z^\chi_\mu)^*\otimes Z^\chi_\mu$ corresponding to the identity matrix. Let $c^\chi_\mu\in \mathcal{U}_0(\mathfrak{n})\otimes\mathcal{U}_0(\mathfrak{n}_-)$ be the corresponding element under the above isomorphism. Note that as a $\mathcal{U}_0(\mathfrak{n})$-module, we have an isomorphism $(Z^\chi_\mu)^*\otimes Z^\chi_\mu\cong \mathcal{U}_0(\mathfrak{n})\otimes Z^\chi_\mu$ using the above isomorphism. Now for $v\in V$ we get a $\mathcal{U}_0(\mathfrak{n})$ submodule generated by $v\otimes  Z^\chi_\mu\subset V\otimes  Z^\chi_\mu$, which is a quotient of $\mathcal{U}_0(\mathfrak{n})\otimes Z^\chi_\mu$ and thus $c^\chi_\mu(v\otimes 1^\chi_\mu)$ is an invariant $\mathcal{U}_0(\mathfrak{n})$-submodule of $V\otimes  Z^\chi_\mu$, ie a highest weight vector. We thus get a linear map
	$$V^{(\lambda,\lambda')}_{\mu-\mu'}\rightarrow Hom_{\mathcal{U}_\chi(\mathfrak{g^{(r-1)}})-mod^{gr}}(Z^\chi_\mu,Z^\chi_{\mu'}\otimes V^{(\lambda,\lambda')}) $$
	 defined by $v\mapsto(1^\chi_\mu\mapsto c^\chi_{\mu'}(v\otimes 1^\chi_{\mu'}))$.We check this is injective. Note that we have a linear map $V\otimes Z^\chi_\mu\rightarrow V\otimes k\cong V$ given by quotenting the spaces of weight smaller than $\mu$. Using the isomorphism $V\otimes k\cong V$ given by $v\otimes 1^\chi_\mu\mapsto v$ we get that the above highest weight vectors $c^\chi_\mu(v\otimes 1^\chi_\mu)$ goes to $v$ under the above map. It follows that the space of highest weight vectors described by the above map has a left inverse so is injective.\\
	Further $V\otimes Z^\chi_\mu$ has a filtration with associated grades $\oplus V_{\mu-\mu'}\otimes Z^\chi_{\mu'}$ and because $\chi$ is generic this category is semisimple, thus $V\otimes Z^\chi_\mu\cong \oplus V_{\mu-\mu'}\otimes Z^\chi_{\mu'}$ and further it follows that the dimension of highest weight vectors of weight $\mu'$ are exactly $dim(V^{(\lambda,\lambda')}_{\mu-\mu'})$ and so the above map is an isomorphism.\\
	Now using this isomorphism we will see what the composition law gives. From the above computation we get the composition of the maps given by $v\in V_{\mu-\mu'}$ and $w\in W_{\mu'-\mu''}$ is described by the highest weight vector $c^\chi_{\mu'}(v\otimes c^\chi_\mu(w\otimes 1^\chi_\mu))$, where $c^\chi_{\mu'}$ acts via $id\otimes \Delta$, where $\Delta$ is the coproduct. We saw above that projecting to the highest weight space of $Z^\chi_{\mu''}$ gives an inverse to the above isomorphism, and thus we get the above highest weight vector comes from the vector $c^\chi_{\mu'}(v\otimes w)$.\\
	Thus we get an isomorphism of the underlying vector spaces of both algebras, whose products differ by the twist given by multiplication with $c^\chi_\mu$ for an appropriate $\mu$.	
\section{Computations for $Sl_2$}
In this section we compute the algebra governing the $SL_2$ representations for a graded Frobenius kernel at a $0$ parameter, up to some constants, to see the above relation gives actually two isomorphic algebras, so the above relation becomes an equivalence of categories.\\
We then further compute the center of the graded Frobenius kernel for $SL_2$
\subsection{Presentation of the algebra}
In this first part we find a presentation, up to some constants, which we do not compute.\\
Note that the projective module $P_i:=_1P_i$ for the first Frobenius kernel considered as a $G$ module has subquotients given by $L_i$, $L_1^{(1)}\otimes L_{p-2-i}$ and $L_i$, both as the socle and cosocle filtration, when $0\leq i\leq p-2$ and $P_{p-1}=L_{p-1}$.\\
Using this we find some morphisms, which we will prove generate and we will find the relations.\\
We will break up maps into partial maps acting separately on each tensor factor. To do this we first introduce the single level map, ie we introduce some maps of $G$-modules involving $P_i$. To do this we introduce the notation $V=L_1$ as we will use this repeatedly.\\
Note that $P_i\otimes V\cong P_{i-1}\oplus P_{i+1}$ for $1\leq i\leq p-3$, $P_0\otimes V\cong V^{(1)}\otimes P_{p-1}\oplus P_1$, $P_{p-2}\otimes V\cong P_{p-3}\oplus P_{p-1}^{\oplus 2}$ and $P_{p-1}\otimes V\cong P_{p-2}$.\\
We fix splittings $P_r\rightarrow P_{r\pm1}\otimes V\rightarrow P_r$ for $0\leq r, r\pm 1\leq p-2$ and also a splitting $V^{(1)}\otimes P_{p-1}\rightarrow P_0\otimes V$, which exists by the above. Further fix an isomoprhism $P_{p-1}\otimes V\cong P_{p-2}$. Using this we can find by adjunction of the above maps, splittings $P_{r+1}\rightarrow P_r\otimes V\rightarrow P_{r+1}$.\\
We also fix maps $V^{(1)}\otimes P_r\rightarrow P_{p-2-r}$ for $0\leq r\leq p-2$. These can be seen to exist by an easy computation in the second frobenius kernel of $SL_2$.\\
Now we will write some commuting diagrams for the above maps, where we write a commuting diagram to mean the maps agree up to a non-zero constant multiple of each other.\\
For $0\leq r, r\pm 1\leq p-2$ these map then satisfy the following relations:
\[\begin{tikzcd}
V^{(1)}\otimes P_r \arrow[r] \arrow[d]
& P_{p-2-r} \arrow[d] \\
V^{(1)}\otimes P_{r\pm 1}\otimes V \arrow[r]
& P_{p-2-(r\pm 1)}\otimes V
\end{tikzcd}\] 
These are satisfied as, by considering the block decomposition, these two have to have the same image and the endomorphisms of the image are given by just constants.\\
Also we have the following commutative diagram defining $\Omega:P_r\rightarrow P_r$ for $0\leq r\leq p-2$:
\[\begin{tikzcd}
V^{(1)}\otimes V^{(1)}\otimes P_r \arrow[r] \arrow[d]
& P_r \arrow[d,"\Omega"] \\
V^{(1)}\otimes P_{p-2-r} \arrow[r]
& P_r
\end{tikzcd}\] 
Here the top map is the duality pairing $V\otimes V\rightarrow L_0$. The map $\Omega$ can also be factored as $P_i\rightarrow L_i\rightarrow P_i$ corresponding to these being both injective hulls and projective covers of $L_i$ in the first Frobenius kernels. Note that this and $Id$ span all endomorphisms of $P_i$.\\
We also define two maps $\phi_{min/max}:P_{p-1}\rightarrow P_{p-2}\otimes V$, by the commuting diagram:
\[\begin{tikzcd}
P_{p-2} \arrow[r] \arrow[d, "\Omega",shift left=1.5ex]\arrow[d,"Id",shift right=1.5ex ]
& P_{p-1}\otimes V \arrow[d,"\phi_{min}"] \arrow[d,"\phi_{max}",shift right=5.5ex] \\
P_{p-2} \arrow[r,leftarrow]
& P_{p-2}\otimes V\otimes V
\end{tikzcd}\] 
and by the condition:
\[\begin{tikzcd}
P_{p-1} \arrow[r] \arrow[d]
& P_{p-2}\otimes V \arrow[dl] \\
P_{p-1}\otimes V 
\end{tikzcd}\] 
We can do this as the two maps $P_{p-1}\rightarrow P_{p-2}\otimes V$ come from the socle and head $L_{p-2}$ in $P_{p-2}$ and from the description of $\Omega$ above. Further as $\phi_{min}$ corresponds to the socle of $P_{p-2}$ we get the commutativity:
\[\begin{tikzcd}
V^{(1)}\otimes V^{(1)}\otimes P_{p-1} \arrow[r] \arrow[d]
& P_{p-1} \arrow[d,"\phi_{min}"] \\
V^{(1)}\otimes P_0\otimes V \arrow[r]
& P_{p-2}\otimes V
\end{tikzcd}\] 
Further  by a similar argument we get:
\[\begin{tikzcd}
V^{(1)}\otimes P_{p-1} \arrow[dr] \arrow[d, "\phi_{max}"]\\
V^{(1)}\otimes P_{p-2}\otimes V \arrow[r]
& P_0\otimes V
\end{tikzcd}\]
And lastly we get the commutative diagram given by 
\[\begin{tikzcd}
P_{p-1} \arrow[r,"\phi_{min}", shift right=1.5ex]\arrow[r,shift left=1.5ex, "\phi_{max}"] \arrow[d, "Id"]
& P_{p-2}\otimes V \arrow[d] \\
P_{p-1} \arrow[r,leftarrow]
& P_{p-1}\otimes V\otimes V
\end{tikzcd}\] 
\\
Using these we define the morphisms of projectives in Frobenius kernels as follows:
\[\begin{tikzcd}
\dots P_{k_l\pm 1}^{(l)}\otimes P_{k_{l-1}}^{(l-1)}\dots \arrow[r] \arrow[d]
& \dots P_{k_l}^{(l)}\otimes P_{p-2-k_{l-1}}^{(l-1)}\dots \\
\dots P_{k_l}^{(l)}\otimes V^{(l)}\otimes P_{k_{l-1}}^{(l-1)}\dots \arrow[ur]
\end{tikzcd}\]  
Here we have $0\leq k_{l-1}\leq p-2$ and $0\leq k_l\pm 1\leq p-2$ and $0\leq k_l\leq p-1$. The action on the rest of the factors is given by the identity. We say this morphism has level $l-1$. We will define several more morphisms, but we will not write the dots, and we will assume the action on the missing factors is given by the identity.\\
We continue defining maps as follows:
\[\begin{tikzcd}
P_{p-1}^{(l)}\otimes P_{k_{l-1}}^{(l-1)} \arrow[r] \arrow[r, shift right=1.5ex] \arrow[d] \arrow[d, shift right=1.5ex]
& P_{p-2}^{(l)}\otimes P_{p-2-k_{l-1}}^{(l-1)} \\
P_{p-2}^{(l)}\otimes V^{(l)}\otimes P_{k_{l-1}}^{(l-1)} \arrow[ur]
\end{tikzcd}\] 
Again we say this map has level $l-1$ and the two maps are the maps $\phi_{min/max}$.\\
We check these maps generate and further that every map is a linear combination of monomials in these maps, where the product is taken as a product of maps of non-decreasing level.\\
To check generation we proceed by induction. From the filtration considered above for $P_i$, we get a filtration on $P_{k_r}^{(r)}\otimes\dots P_{k_0}$ with associated graded given by $P_{k_r}^{(r)}\dots P_{k_2}^{(2)}\otimes ((P_{k_1}^{(1)})^{\oplus 2}\otimes L_{k_0}\oplus (P_{k_1+1}\oplus P_{k_1-1})\otimes L_{p-k_0-2})$. Here we assume $k_0\neq p-1$ and $1\leq k_1\leq p-3$. We have similar decompositions for other $k_1$ given by the decompositions of $P_{k_1}\otimes V$. Note that if $k_1=0$, we get a summand $V^{(2)}\otimes P_{p-1}^{(1)}$ of $P_0^{(1)}\otimes V^{(1)}$. This extra $V^{(2)}$ give a decomposition of $P_{k_2}^{(2)}\otimes V^{(2)}$ and we can continue inductively if $k_2=0$. Since these are projectives modules we just need to find linearly independent corresponding to the irreducible subquotients. We have maps given by $z\otimes id, z\otimes \Omega: P_{k'_l}^{(l)}\dots P_{k'_1}^{(1)}\otimes P_{k_0}\rightarrow  P_{k_l}^{(l)}\dots P_{k_1}^{(1)}\otimes P_{k_0}$ corresponding to the two summands ending in $L_{k_0}$ in the above associated graded. Further composing a map $z\otimes 1$ with the map $ P_{k_l}^{(l)}\dots P_{k_1\pm 1}^{(1)}\otimes P_{p-2-k_0}\rightarrow  P_{k_l}^{(l)}\dots P_{k_1}^{(1)}\otimes P_{k_0}$ described above, we get the subquotients corresponding to the summands of the associated graded ending in $L_{p-2-k_0}$. The same argument works for all the other decompositions of $P_i\otimes V$ as we have similarly defined maps, except for the case $P_0\otimes V$.\\
For the remaining case we need to introduce the following map:
\[\xymatrix{
P_{k_l\pm 1}^{(l)}\otimes P_{p-1}^{(l-1)}\dots P_{k_0} \ar[r] \ar[d]
& P_{k_l}^{(l)}\otimes P_0^{(l-1)}\dots P_{p-2-k_0}\\
P_{k_l}^{(l)}\otimes V^{(l)}\otimes P_{p-1}^{(l-1)}\dots P_{k_0} \ar[d]&\\
P_{k_l}^{(l)}\otimes P_0^{(l-1)}\otimes V^{(l-1)}\dots P_{k_0} \ar[d]&\\
\vdots \ar[d]&\\
P_{k_l}^{(l)}\otimes P_0^{(l-1)}\dots V^{(1)}\otimes P_{k_0}\ar[uuuur]    \\
}\] 
Note that we can factor the map $V^{(1)}\otimes P_{p-1}\rightarrow P_0\otimes V$ through $V^{(1)}\otimes P_{p-2}\otimes V$ as noted above. And hence we can factor this missing map through a sequence of maps already defined. The only missing case is if $k_0=p-1$ in this case note that taking Frobenius twist and tensoring by the Steinberg is an equivalence of categories onto the block over the Steinberg. It follows that the maps over the Steinberg are generated by the above maps by induction on the Frobenius kernels. Thus we get by the above argument and induction that the above described maps generate the endomorphism algebra governing the category.\\
Now we want to commute two of the above maps if they are of different level. Note that if the level differs by two, the maps act on different factors so they clearly commute with each other. So it remains to check a commutativity condition for maps differening in level by one. To do this we consider the following diagram:
\[
	\small\begin{tikzcd}[column sep= tiny]
P_{s\pm 1}^{(2)}\otimes P_{r\pm 1}^{(1)}\otimes P_k \arrow[r] \arrow[d]
& P_{s\pm 1}^{(2)}\otimes P_{r}^{(1)}\otimes V^{(1)}\otimes P_k \arrow[r] \arrow[d]
& P_{s\pm 1}^{(2)}\otimes P_r^{(1)}\otimes P_{p-2-k} \arrow[d]\\
P_s^{(2)}\otimes V^{(2)}\otimes P_{r\pm 1}^{(1)}\otimes P_k \arrow[r] \arrow[d]
& P_s^{(2)}\otimes V^{(2)}\otimes P_r^{(1)}\otimes V^{(1)}\otimes P_k \arrow[r] \arrow[d]
& P_s^{(2)}\otimes V^{(2)}\otimes P_r^{(1)}\otimes P_{p-2-k} \arrow[d]\\
P_s^{(2)}\otimes P_{p-2-(r\pm 1)}^{(1)}\otimes P_k \arrow[r] 
& P_s^{(2)}\otimes P_{p-2-r}^{(1)}\otimes V^{(1)}\otimes P_k \arrow[r]
& P_s^{(2)}\otimes P_{p-2-r}^{(1)}\otimes P_{p-2-k} \\
\end{tikzcd}\] 
Here every square is obviously commutative except the bottom left which commutes up to non-zero factor by the statements made above. We also have the following commutativity.
\[\small\begin{tikzcd}[column sep=tiny]
P_{s\pm 1}^{(2)}\otimes P_{p-1}^{(1)}\otimes P_k \arrow[r] \arrow[d]
& P_{s\pm 1}^{(2)}\otimes P_{p-2}^{(1)}\otimes V^{(1)}\otimes P_k \arrow[r] \arrow[d]
& P_{s\pm 1}^{(2)}\otimes P_{p-2}^{(1)}\otimes P_{p-2-k} \arrow[d]\\
P_s^{(2)}\otimes V^{(2)}\otimes P_{p-1}^{(1)}\otimes P_k \arrow[r] \arrow[d]
& P_s^{(2)}\otimes V^{(2)}\otimes P_{p-2}^{(1)}\otimes V^{(1)}\otimes P_k \arrow[r] \arrow[dl]
& P_s^{(2)}\otimes V^{(2)}\otimes P_{p-2}^{(1)}\otimes P_{p-2-k} \arrow[d]\\
P_s^{(2)}\otimes P_0^{(1)}\otimes P_k) \arrow[rr] 
&
& P_s^{(2)}\otimes P_0^{(1)}\otimes P_{p-2-k} \\
\end{tikzcd}\] 
Here again we use only need to check the bottom left triangle and this is one of the triangles given above. \\
From this we can move maps past each other if they have decreasing level. The remaining relations are relations between maps of same level.\\
We have the following commutative diagram
\[\begin{tikzcd}
P_{r\pm 2/0}^{(1)}\otimes P_{k} \arrow[r] \arrow[ddr, "\theta", bend right]
& P_{r\pm 1}^{(1)}\otimes V^{(1)}\otimes P_{k}^{(1)}\arrow[r] \arrow[d]
& P_{r\pm 1}^{(1)}\otimes P_{p-2-k} \arrow[d]\\
& P_r^{(1)}\otimes V^{(1)}\otimes V^{(1)}\otimes P_k \arrow[r] \arrow[d]
& P_r^{(1)}\otimes V^{(1)}\otimes P_{p-2-k} \arrow[d]\\
&P_r^{(1)}\otimes P_k \arrow[r,"\Omega"] 
& P_r^{(1)}\otimes P_k \\
\end{tikzcd}\]
Here the square on the right comes form the commutative diagram
\[\begin{tikzcd}
P_{r\pm 2/0} \arrow[r] \arrow[d,"\theta"]
& P_{r\pm 1}\otimes V \arrow[d]  \\
P_r \arrow[r,leftarrow]
& P_r\otimes V\otimes V
\end{tikzcd}\] 
This commutes in the case $s\pm 2$ with $\theta$ being the zero map, just looking at the blocks. In the remaining case we could have chossen the splitting in such a way that the two maps were adjoint under the self-adjunction of the functor of tensoring with $V$. In that case, the square exactly presents that description. Any other choice of splitting differs hence by an automorphism of $P_s$, thus the above indeed commutes with some automorphism $\theta$ that depends on the above fixed choice of splitting. Note that we can not choose both the splittings to be simultaneously adjoint to each other. This can be seen by noting that the composition of adjoints is not the identity. This is exactly what makes the relations slightly more complicated and why we have chosen to describe the relations up to non-zero constants, or in this case up to an isomorophism, which still allows us to prove our goals, but reduces the computations. The constants and isomorphisms above can be computed, but this requires some algebra we prefer to skip.\\
The above commuting diagram works in the case of the zero map always, as long as all $0\leq s\pm 2,s\leq p-1$ and in the other case when both $s$ and $s\pm 1$ are $\neq p-1$. The only remaining cases are the ones involving only $p-1$ and $p-2$.\\
The first of these commuting diagrams is given by 
\[\begin{tikzcd}
P_{p-1}^{(1)}\otimes P_{k} \arrow[r,shift right=1.5ex,"\phi_{min}"] \arrow[r, shift left=1.5ex,"\phi_{max}"] \arrow[ddr, "Id",bend right]
& P_{p-2}^{(1)}\otimes V^{(1)}\otimes P_{k}^{(1)}\arrow[r] \arrow[d]
& P_{p-2}^{(1)}\otimes P_{p-2-k} \arrow[d]\\
& P_{p-1}^{(1)}\otimes V^{(1)}\otimes V^{(1)}\otimes P_k \arrow[r] \arrow[d]
& P_{p-1}^{(1)}\otimes V^{(1)}\otimes P_{p-2-k} \arrow[d]\\
&P_{p-1}^{(1)}\otimes P_k \arrow[r,"\Omega"] 
& P_{p-1}^{(1)}\otimes P_k \\
\end{tikzcd}\]
and the last commuting diagram is given by 
\[\begin{tikzcd}
P_{p-2}^{(1)}\otimes P_{k} \arrow[r] \arrow[ddr, "Id", bend right] \arrow[ddr, "\Omega", shift right=3ex, bend right]
& P_{p-1}^{(1)}\otimes V^{(1)}\otimes P_{k}^{(1)}\arrow[r] \arrow[d, "\phi_{min}",shift right=3ex]\arrow[d,shift left=3ex, "\phi_{max}"]
& P_{p-1}^{(1)}\otimes P_{p-2-k} \arrow[d]\\
& P_{p-2}^{(1)}\otimes V^{(1)}\otimes V^{(1)}\otimes P_k \arrow[r] \arrow[d]
& P_{p-2}^{(1)}\otimes V^{(1)}\otimes P_{p-2-k} \arrow[d]\\
&P_{p-2}^{(1)}\otimes P_k \arrow[r,"\Omega"] 
& P_{p-2}^{(1)}\otimes P_k \\
\end{tikzcd}\]
Again here all squares clearly commute, except the square on the left, which commutes by the definition of $\phi_{min/max}$ and the properites of these stated above.\\
There are also relations of the maps of level $r$, but those can be argued similarly to the above, but easier, as there is only an action on one tensor factor instead of two.\\
This in fact give all the conditions up to constants. This follows as the above conditions allow us to first describe every element as a linear combination of maps composed in increasing order. Further the elements of one fixed level generate an algebra, which is spanned by the above generators and also by the $\Omega$'s and the idempotents of each projective factor of level l. This follows by the description of the relations in the first Frobenius kernels. As we can describe the projective modules and the simple modules as tensor products, we get that this in fact give exactly as many linearly independent elements as simple subquotients of a fixed projective, hence give precisely a basis.\\
The above description of the algebra can be upgraded with further work to an exact description by generators and relations. To do this we have to choose the above splittings $P_r\rightarrow P_{r+1}\otimes V\rightarrow P_r$ and then considering the adjoint maps of this compute the composition $P_{r+1}\rightarrow P_r\otimes V\rightarrow P_{r+1}$. This gives some automorphism of $P_{r+1}$ and the missing constants and autmorphisms exactly arise from this. The above proof describes the algebra with precise generators and relations up to the computations of these automorphisms. This also is not a very complicated computation, but would make the proof a bit more messy, so we have chosen to skip these details.
\subsection{Equivalence in the case of $Sl_2$}
In this subsection we prove that the above relation between the algebras governing the categories $\mathcal{U}^{[r]}_0(\mathfrak{g})-mod^{gr}$ and $\mathcal{U}^{[r+1]}_\chi(\mathfrak{g})-mod^{gr}$ that differ by the above given twist, are actually isomorphic algebras.\\
\begin{lem}
	The twists $c^\chi_n$ for $SL_2$ are given by $\sum_{k=0}^{p-1}\frac{(-1)^k}{k!d^\chi_n(d^\chi_n-1)\dots(d^\chi_n-k+1)}e^k\otimes f^k$
\end{lem}
\begin{proof}
	Note that the definition of $c^\chi_n$ is given by understanding the trivial submodule of $(Z^\chi_n)^*\otimes Z^\chi_n$. Note that $(Z^\chi_n)^*$ is generated by a lowest weight vector $1^{-\chi}_{-n}$ and $Z^\chi_n$ is generated by a highest weight vector $1^\chi_n$, in such a way that these two pair to 1.\\
	Thus by looking at the weight, we get the canonical pairing is given by an element of the form 
	$$c^\chi_n(1^{-\chi}_{-n}\otimes 1^\chi_n)\sum_{k=0}^{p-1}(A^\chi_n)_k e^k1^{-\chi}_{-n}\otimes f^k1^\chi_{n}$$
	for some constants $(A^\chi_n)_k$.\\
	Here $(A^\chi_n)_0=1$. Further we need the action of $\mathfrak{g}$ on $c^\chi_n(1^{-\chi}_{-n}\otimes 1^\chi_n)$ to be trivial. So checking the condition for $e$ to act trivially, we get the equations
	$$(A^\chi_n)_{k-1}+k(d^\chi_n-k+1)(A^\chi_n)_k=0$$
	Here $d^\chi_n$ is the constant by which $h$ acts on $1^\chi_n$. Note in particular, as $\chi$ is a generic semisimple element, $d^\chi_n$ is not an integer, so we can indeed divide by $d^\chi_n-r$ for $r\in \mathbb{Z}$, so we indeed get the solution as stated in the lemma.
\end{proof}
Now with this explicit twist, we can prove that when applied to the above algebra, it produces an isomoprhic algebra. Hence we get the folowing result.
\begin{thm}
	For $SL_2$ and for $\chi$ a generic semisimple element we have an equivalence of categories
	$$\mathcal{U}^{[r]}_0(\mathfrak{g})-mod^{gr}\cong\mathcal{U}^{[r+1]}_\chi(\mathfrak{g})-mod^{gr}$$
\end{thm}
\begin{lem}
	The algebra described in the previous section and its twist by $c^\chi_n$ given above, are isomorphic.
\end{lem}
\begin{proof}
	Note that the algebra given above only has $G$ representations isomorphic to $L_0$ and $L_1$. Further note that the product of some element of $L_0$ and something in any other representation is unchanged after the twist, since any higher term in $\sum_{k=0}^{p-1}\frac{(-1)^k}{k!d^\chi_n(d^\chi_n-1)\dots(d^\chi_n-k+1)}e^k\otimes f^k$ acts trivially on $L_0$.\\
	Hence the only products that change are products of elements in $L_1$.\\
	Note the above generators are part of an $L_1$ for the map $V^{(r+1)}\otimes P_{k_r}^{(r)}\otimes\dots P_{k_0}\rightarrow P_{p-k_r-2}^{(r+1)}\otimes\dots P_{k_0}$ for $k_r\in\Lambda^1$. So only the generators of maximal level $r$ have a non-trivial action of $G$. The other generators are fixed by the action of $G$.\\
	Further, denote the basis of $L_1$, with corresponding weight, by $v_{-1}$ and $v_1$. Denote the composition in the untwisted product by $\circ$ and in the twisted product by $\tilde{\circ}$. Then 
	\begin{enumerate}\item $v_1\tilde{\circ} v_{-1}=v_1\circ v_{-1}$
	\item $v_{-1}\tilde{\circ} v_{1}=v_{-1}\circ v_{1}+(A_n^\chi)_1 ev_{-1}\circ fv_{1}=(1-(A_n^\chi)_1)v_{-1}\circ v_{1}$
	\end{enumerate}, where $n$ depends between which graded piece the map is given, just as the dependence of the twist described in section 3. Note that as $\chi$ generic $(1-(A_n^\chi)_1)$ is a non-zero constant.\\
	To continue first introduce some notation for the algebras. Let $A=End(P)$ and $A'=End(P')$, with $P$ and $P'$ given by the notation of section 3. We have described above a set of generators and relations of $A$. We will pick some generators in $A'$ and we will see they are subject to the same relations as $A$ and hence we get an isomorphism.\\
	To do this we will call the generators $k^{(r+1)}_n\otimes P_{k_r}^{(r)}\otimes\dots \otimes P_{k_0}\rightarrow k^{(r+1)}_{n\pm 1}\otimes P_{p-k_r-2}^{(r+1)}\otimes\dots P_{k_0}$, the generators of level $r$ and degree $n$. The generator corresponding to the $+$ in $\pm$ will de called level $r$ degree $n$ increasing generator and the other one will be decreasing.\\
	For the generators of $A'$ we take exactly the same generators as for $A$, except we scale the level $r$ generators. Namely, scale the level $r$ degree $n$ generators by non-zero constants $D^{\pm,\chi}_n$ depending if it is increasing or decreasing.\\
	These are still generators, as the product and the relations of the elements of degree less than $r$ remains unaltered and by the rule of composition (1), we have that the $\Omega's$ of level $r$ are still generated and thus we get everything.\\
	Now the product between two generators of level less than $r$ remains unaltered, so the relations remain the same. The product between a generator of level $r$ and one of level less than $r$ also remains unaltered. The relations satisfied by these are given by commutative squares where two parallel arrows are of level $r$ both of the same degree and both either decreasing of increasing and some generator of smaller level. Since we scale all generators of level $r$ with same degree and increasing or decreasing by the same constant, we get that this relation is still satisfied, as we have just scaled both sides by the same constant.\\
	The only remaining relations to check are the ones given by two generators of level $r$. The composition of two increasing or two decreasing generators is always zero and remains so after the twist. So the only missing case is the composition of one increasing and one decreasing generator.\\
	Denote by $\psi_n^\pm:k^{(r+1)}_n\otimes P_{(k_r)}^{(r)}\otimes\dots k^{(r+1)}_{n\pm 1}\otimes P_{k_0}\rightarrow P_{p-k_r-2}^{(r+1)}\otimes\dots P_{k_0}$ these generators. The only missing relation to check is that $\psi_{n+1}^-\circ\psi_n^+$ and $\psi_{n-1}^+\circ\psi_n^-$ are equal up to a constant. This constant depends on our choice of maps of one Frobenius kernel, used to define the maps. We need to check that after rescaling they still satisfy the same equation with the same constant. To do this from the condition of twisted composition we just need to satisfy the following equation
	$$D^{+,\chi}_{n-1}D^{-,\chi}_n=D^{-,\chi}_{n+1}D^{+,\chi}_n(1-(A_{n+1}^\chi)_1)$$
	Note that as $(1-(A_{n+1}^\chi)_1)$ are non-zero, we can choose non-zero $D^{\pm,\chi}_n$ to satisfy the above equation. Thus we can choose generators of $A'$ such that they satisfy the same relations as the generators of $A$ and thus these are isomorphic algebras.\\
	The theorem follows immediately from this lemma.
\end{proof}
\subsection{Center of Frobenius kernels of $SL_2$}
In this subsection we compute the center of the Frobenius kernels of $SL_2$.\\
\begin{thm}
	The center of the rth Frobenius kernel of $SL_2$ over a regular block is spanned by the elements
	$$1\otimes 1\dots 1\otimes e_{block}\otimes\Omega_{k_l}\otimes\Omega_{k_{l-1}}\otimes\dots\Omega_{k_0}$$
	and the idempotent of the block. Here $e_{block}$ is an idempotent of a block of the first Frobenius kernel.
\end{thm}
\begin{rmk}
	That is the center is given by endomorphisms of the projective, that act by some $\Omega$ on the lowest levels and acts by $1$ on the first factors.\\
	From this you can describe the full center, as the only singular block is over the Steinberg, and recall that the functor $M\mapsto M^{(1)}\otimes St$ is an equivalence of $Rep(G)$ with the block over the Steinberg. Using this inductively we can understand the full center of $SL_2$.
\end{rmk}
\begin{proof}
	First we check that these elements are central.\\
	 Note that if you compose the above element on either side with an element of level $\leq l$ we get 0. This is because pre- or postcomposing with $\Omega$ the map $V^{(1)}\otimes P_k\rightarrow P_{p-2-k}$ gives 0.\\
	 Composing with the above generators of level higher than $l$ is obviously commutative, as both maps act on different tensor functors.\\
	 We thus get the above maps are clearly central, as the idempotent of the block is obviously central.\\
	 Now we will prove these are the only central transformations. To do this we proceed by induction.\\
	 First we describe the center of the first Frobenius kernel. To do this, we can break up the algebra into blocks. For a single block we have $P_{k+2rp}$ and $P_{2rp-k-2}$ are all the projectives. And the only remaining cases are to check if an element $z$ given by a linear combinations of idempotents $\sum \mu_{k+2rp}e_{k+2rp}+\mu_{2rp-2-k}e_{2rp-2-k}$ is central, since all $\Omega$'s are central and these are all endomorphisms of the indecomposable projectives. Consider pre- and postcomposing $z$ with the map $P_{k+2rp}\rightarrow P_{2rp-k-2}$. This forces $\mu_{k+2rp}=\mu_{2rp-2-k}$. Similarly considering $P_{k+2rp}\rightarrow P_{2(r+1)p-k-2}$, we get $\mu_{k+2rp}=\mu_{2(r+1)p-2-k}$. Thus we have $z$ is just a scalar factor of the identity. And hence the center for the first Frobenius kernel is as described above.\\
	 Now we assume the above describes the center of Frobenius kernels up to the $r-1$st Frobenius kernel. Then for the $r$th Frobenius kernel, note that all the endomorphisms of a single indecomposable projective is given by a tensor product of maps $e_k$ and $\Omega_r$. Hence in particular it decomposes as a tensor product of maps only acting on one tensor factor. Thus we can decompose any sum of endomorphisms of these indecomposable as $z=\sum z_k\otimes e_k+y_k\otimes \Omega_k$, for some maps $z_k$ and $y_k$.\\
	 Now assume this element $z$ is central. Then composing with the generators of level higher than 0, we must have $z_k$ and $y_k$ must be central maps of the $r-1$st Frobenius kernel. By induction the $y_k$ are of the above description and hence using the above elements which we already know are central, we just need to consider the elements of the form $\sum z_k\otimes e_k$.\\
	 Suppose $z_k\otimes e_k$ acts by a non-zero transformation on the indecomposable projective $P_{k_r}^{(r)}\otimes\dots P_{k_1}^{(1)}\otimes P_{k_0}$. Such a projective exists unless $z_k=0$. We will denote $k_1=r$ and $k_0=k$ and denote the higher level factors through dots.\\
	 We will now consider the relation of commutativity with the map $\dots P_r^{(1)}\otimes P_k\rightarrow\dots P_{r+1}^{(1)}\otimes P_{p-2-k}$. We know the map $P_r\rightarrow P_{r+1}\otimes V$ used to define this map is injective and the above $z$ acts by the idempotent on $P_k$. Thus precomposing the above with $z$ gives a non-zero map. It follows that the action of $z$ on $\dots P_{r+1}^{(1)}\otimes P_{p-2-k}$ is non-zero.\\
	 Continuing like this we can assume for $r=p-2$, the action is non-trivial.\\
	 Now we can consider the commutativity with the map $\phi_{max}:\dots P_{p-1}^{(1)}\otimes P_{p-2-k}\rightarrow \dots P_{p-2}^{(1)}\otimes P_k$. We can take away multiples of the idempotent of the block to assume that $z_k$ is a linear combination of central elements that act by $\Omega$ on the last factor $P_{p-2}$. Composing $\Omega$ and $\phi_{max}$ give $\phi_{min}$, so this composition with $z$ only contributes functions of the form $g\otimes \phi_{min}$. But on $\dots P_{p-1}^{(1)}\otimes P_{p-2-k}$ the central element $z_k$ can only act non-trivially on higher level terms, so the composition is $f\otimes \phi_{max}$, since $\phi_{max}$ only acts on the bottom two tensor factors. Thus we must have $f=0$, $g=0$, in order to get commutativity, by the description of the relations above. It follows hence that $z$ acts by $0$ on $\dots P_{p-2}^{(1)}\otimes P_k$. this contradict the above. Hence we get a contradiction to the non-zero action on the original projective indecomposable. Note that here, we have taken away a scalar multiple of an idempotent of a block, so up to some elements appearing as summands of these idempotents all others do not appear. In other words, we can reduce $z$ to be a linear combination of tensor products of $e_k$. So these are the idempotents of the indecomposable projectives. But if we have a central element $z$ given as a linear combinations of idempotents, then they have to have the same factor for two indecomposable projective $P$ and $Q$, if there is a non-zero map $P\rightarrow Q$. This precisely gives the idempotents of a block.\\
	 From the above argument hence, we see that every central element is a linear combination of the above and thus the center is described as above.
\end{proof}
%Note that for $1\leq k_1\leq p-3$  and $k_0\neq p-1$ there is a filtration of $P(k_1)^{(1)}\otimes P(k_0)$ with associated graded given by $P(k_1)^{(1)}\otimes L(k_0)\oplus P(k_1-1)^{(1)}\otimes L(p-2-k_0)\oplus P(k_1+1)^{(1)}\otimes L(p-2-k_0)\oplus P(k_1)^{(1)}\otimes L(k_0)$, coming from the above filtration of $P(k)$ given above. Hence as these are projective we get maps 
%$$\theta^\pm_{(k_i)_{i=0}^1}: P(k_1\pm 1)^{(1)}\otimes P(p-2-k_0)\rightarrow P(k_1)^{(1)}\otimes P(k_0)$$
%By similarly consider this for $k_1=0,p-2 or p-1$ and $k_0\neq p-1$ we get the following maps $\theta_{(p-1,k_0)}: P(p-2)^{(1)}\otimes P(p-2-k_0)\rightarrow P(p-1)^{(1)}\otimes P(k_0)$,
%$\theta^-_{(p-2,k_0)}: P(p-3)^{(1)}\otimes P(p-2-k_0)\rightarrow P(p-2)^{(1)}\otimes P(k_0)$
%$\theta^{+,1/2}_{(p-2,k_0)}: P(p-1)^{(1)}\otimes P(p-2-k_0)\rightarrow P(p-2)^{(1)}\otimes P(k_0)$ and\\
%For $k_0=p-1$ this is the Steinberg, so we just get the morphims of the first Frobenius kernel tensored by the identity of the Steinberg.\\
%For the rth Frobenius kernel, we see from all the above we get a unique 

\end{document}